\documentclass[11pt]{amsart}
\usepackage{amscd,amssymb}
\usepackage{graphicx}
\theoremstyle{plain}

\newtheorem{lemma}{Lemma}[section]
\newtheorem{theorem}[lemma]{Theorem}
\newtheorem{prop}[lemma]{Proposition}
\newtheorem{cor}[lemma]{Corollary}

\newtheorem{definition}[lemma]{Definition}

\newcommand{\bbA}{{\mathbb{A}}}
\newcommand{\bbC}{{\mathbb{C}}}
\newcommand{\bbF}{{\mathbb{F}}}

\newcommand{\bbQ}{{\mathbb{Q}}}
\newcommand{\bbR}{{\mathbb{R}}}
\newcommand{\bbZ}{{\mathbb{Z}}}

\newcommand{\Gal}{{\mathrm{Gal}}}
\newcommand{\Fr}{{\mathrm{Fr}}}

\newcommand{\unit}{{\mathrm{unit}}}

\begin{document}
\title{Computing finite Galois groups arising from automorphic forms}

\author{Kay Magaard and Gordan Savin}
\address{ School of Mathematics, University of Birmingham, Birmingham, B15 2TT, England} 
\email{k.magaard@bham.ac.uk}
\address{Department of Mathematics, University of Utah, Salt Lake 
City, UT 84112}
\email{savin@math.utah.edu}

\maketitle
\vskip 10pt

\section{Introduction} 

Let $p$ be a prime and $G_2(p)$  the exceptional group of type $G_2$ over the finite field of $p$ elements. The goal of this paper is 
 to construct $G_2(p)$  as  a Galois group over $\mathbb Q$ by reducing, modulo $p$, the $p$-adic representation  attached to a regular, self-dual, 
cuspidal  automorphic representation of $GL_7$.  The first step in this direction was taken in \cite{GS} where an algebraic, in the sense of Gross \cite{Gr},
 automorphic representation  $\pi$  was constructed on an anisotropic form of $G_2$. It was also shown that $\pi$ lifts to a cuspidal automorphic representation
  $\Sigma$ on $Sp_6$. In Section 8 we show that $\Sigma$  lifts to a regular, self-dual 
cuspidal automorphic representation  $\Pi$ on $GL_7$ by the recent results of Arthur \cite{Ar}. Since the field of definition of $\Pi$ is $\mathbb Q$, 
 using a result of Taylor \cite{Ta},  in Section 7 we prove that the $p$-adic representation of $\Gal(\bar{\mathbb Q}/\bbQ)$ attached to $\Pi$ is defined over $\mathbb Q_p$. 
We also show that the image of the $p$-adic representation is contained in $G_2(\bbQ_p)$. One expects that $G_2(p)$ appears as a quotient of the image for all but finitely many primes. We show that this is indeed the case for an explicit set of primes of density at least 1/18. The resulting extensions, with the Galois group $G_2(p)$, are unramified at $5$ and $p$ only. The method of the proof is as follows. The components $\Pi_2$ and $\Pi_3$ are unramified and their Satake parameters were computed by Lansky and Pollack in \cite{LP}. In particular, we know conjugacy classes of two elements, the Frobenia at 2 and 3,
 in the image of the $p$-adic representation attached to $\Pi$. Thus we are lead to the following three problems: 1) Developing a notion of reduction, modulo $p$, of rational conjugacy classes. We do this in a generality of split reductive groups in Sections 2-4. 2) Understanding the Galois group of palindromic polynomials. This is 
 the topic of Section 5.   3) Giving a criterion when two elements generate $G_2(p)$. A criterion,  in Section 6, is based on Aschbacher's classification of 
maximal subgroups of $G_2(p)$. This, as well as some other aspects of this paper,  is similar in flavor to a work of Dieulefait \cite{Du}. 

\vskip 5pt 
We benefited from conversations with Ga\"etan Chenevier,  Michael Dettweiler, Wee Teck Gan, Dick Gross, Colette M\oe glin  and Sug Woo Shin. The second author has been supported by a grant from NSF.

\section{Rational semi-simple  conjugacy classes}  

In this section $K$ is any field and $K_s$ a separable closure of $K$. Let $G$ be a connected reductive group split over $K$. 
 We shall fix a faithful algebraic representation $V$ of $G$.  Let $g\in G(K_s)$ be a semi-simple element. 
 Let $R_g(x)\in K_s[x]$ denote the characteristic polynomial of $g$ acting on $V\otimes_K K_s$. 
Let $T$ be a split torus, defined over $K$. Let $\{\chi_i\}$ be the multi-set of weights of $V$ where every weight $\chi$ appears with multiplicity 
$\dim V_{\chi}$. Then, for every $t\in T(K_s)$, 
\[ 
R_t(x)= \prod_{i=1}^n(x-\chi_i(t))
\] 
where $n=\dim(V)$. Let $n_0$ be the dimension of the trivial weight space. We can write 

\[ 
R_t(x)= P_t(x)(x-1)^{n_0}. 
\] 
 Since any semi-simple element $g$ is conjugated to an element in 
$T(K_s)$, the polynomial  $P_g(x)$ is well defined for any semi-simple element and it is an invariant of the conjugacy class $C(K_s)$ of $g$. Hence we also write $P_C(x)=P_g(x)$. 
The class  $C$ is called $K$-rational if  $\sigma(C)=C$ for all $\sigma \in \Gal(K_s/K)$.  In that case $P_C(x)\in K[x]$.  

\vskip 5pt

Let $t\in  C\cap T(K_s)$. Assume that $t$ is strongly regular i.e. the centralizer of $t$ in $G(K_s)$ is $T(K_s)$. Then any  element in $C\cap T(K_s)$ 
 is a conjugate of $t$ by a unique element in the Weyl group $W$. Now assume that $C$ is a $K$-rational conjugacy  class.  Then
$C\cap T(K_s)$ is $\Gal(K_s/K)$-stable, so for every $\sigma \in \Gal(K_s/K)$ there exists a unique element $w\in W$ such that 
$\sigma(t)=t^w$.  Thus, to the $K$-rational conjugacy class $C$ of strongly regular elements we have assigned a homomorphism 
\[ 
\varphi_C: \Gal(K_s/K) \rightarrow W 
\]
unique up to conjugation by $W$. 
Let $E=E_C$ be the finite field extension  of $K$ corresponding to the kernel of $\varphi_C$. 

\begin{prop} Let $C$ be a strongly regular and $K$-rational conjugacy class.
The field $E_C$ is the splitting field of the polynomial $P_C(x)$. 
\end{prop} 
\begin{proof} 
 Since $R_C(x)=\prod_{i=1}^n(x-\chi_i(t))$ and $t$ is determined by the values $\chi_i(t)\in K_s^{\times}$, 
the subgroup of $\Gal(K_s/K)$ fixing the splitting field of $R_C$ (= the splitting field of $P_C$) is indeed the kernel of $\varphi_C$. 
\end{proof}

Note that if $-1\in W$ then $P_t(x)$ is a palindromic polynomial.

\section{Finite tori} 

Assume now that $k$ is a finite field of order $q$. Then $k_s=\bar k$ and $\Gal(\bar k/k)=\langle \Fr_q\rangle $.  Let $T$ be the split torus contained in the split reductive group, as in the previous section, this time defined over $k$. Let $W$ be the Weyl group. For every $w\in W$ let  
\[ 
T_w(k) \subset T(\bar k) 
\] 
be the group of $k$-points in a torus $T_w$ such that  $T_w(\bar k)=T(\bar k)$, but the action of $\Fr_q$ is twisted by $w^{-1}$. 
The order of $T_w(k)$ is given as follows, see \cite{Ca}.  Let $X$ be the lattice of characters of $T$, and $A=X\otimes_{\mathbb Z} \mathbb R$. Let $\Phi_w$ be the characteristic polynomial of $w$ acting on $A$. Then 
 \[ 
 |T_w (k)|=\Phi_w(q). 
 \] 
 
 \vskip 5pt 
 Let us look at the case $G=G_2$. Then $W$ is the dihedral group $D_6$. The conjugacy classes in $W$ are 
\[ 
\begin{array}{|c||c|c|c|c|c|c|}
\hline  C & 1a & 2a & 2b & 2c & 3a & 6a \\
\hline  |C| & 1 & 3 & 3 & 1 & 2 & 2 \\
\hline
\end{array} 
\] 
where the number in the first row is the order of any element in the class and the number 
in the second row is the number of elements in the class. There are 3 classes of elements of 
order 2: the class of reflections about long roots is denoted by 2a, the class of reflections 
about short roots is denoted by 2b, and 2c denotes $-1$. The orders of the corresponding tori 
are 
\[ 
\begin{array}{|c||c|c|c|c|c|c|}
\hline  w & 1a & 2a & 2b & 2c & 3a & 6a \\
\hline  |T_w| & (q-1)^2 & q^2-1 & q^2-1 & (q+1)^2 & q^2+q+1 & q^2-q+1 \\
\hline
\end{array} 
\]

 \section{Reduction mod ${\mathfrak p}$ of conjugacy classes} 
 
 Assume now that $K$ is a number field and $A$ its ring of integers. Assume that we have given to $G$ a structure of an algebraic group scheme over $A$ by 
 fixing a Chevalley lattice $L\subseteq V$ with respect to $T$. In particular, $L$ is a direct sum  of its weight components $L_{\chi}=L\cap V_{\chi}$. 
   For every ring 
 $R$ such that $A\subseteq R \subseteq K_s$ 
 \[ 
 G(R)=\{ g\in G(K_s) ~|~ g \cdot L\otimes_A R= L\otimes_A R\}. 
 \] 
 Since $L$ is a direct sum of  the weight components $L_{\chi}$,  the group  $T(R)$  consists of all elements $t\in T(K_s)$ such that $\chi(t)\in R^{\times}$ for all weights $\chi$ of 
  $V$. 
  
  \vskip 5pt 
 
  Let ${\mathfrak m}$ be a maximal ideal in $R$ let $k_{\mathfrak m} $ be the corresponding residual field. Any element  $g\in G(R)$ acts naturally on 
   the quotient $L\otimes_A k_{\mathfrak m}= L\otimes_A R/L\otimes_A {\mathfrak m}$ as an element in $G(k_{\mathfrak m})$ denoted by $\bar g$. The element $\bar g$ is called the reduction of $g$ modulo ${\mathfrak m}$. 
 
 \vskip 5pt 
 For the remainder of this section we fix $C$,  a strongly regular and $K$-rational conjugacy class.  Let $P_C(x)$ be the corresponding characteristic polynomial. Let $E$ be the finite Galois extension of  $K$ corresponding to the kernel of the homomorphism $\varphi_C:  \Gal(K_s/K) \rightarrow W $ constructed by means of 
 $t\in C\cap T(K_s)$. In particular, $t \in T(E)$. 
 Let $B$ be the ring of integers in $E$. Let $S$ be the set of primes (i.e maximal ideals)  in $A$ such that: 
 \begin{enumerate} 
 \item  the field $E$ is unramified outside $S$. 
 \item $P_C(x) \in A_{\mathfrak p}[x]$ for every ${\mathfrak p}\notin S$, where $A_{\mathfrak p}$ is the localization of $A$ by $A\setminus {\mathfrak p}$.  
 \end{enumerate} 
 
 Let ${\mathfrak q}$ be a prime in $B$  such that ${\mathfrak q}\cap A\notin S$. Let $B_{\mathfrak q}$ be the localization of $B$ by $B\setminus {\mathfrak q}$. 
Since $P_C(x)$ is a monic polynomial, (2) implies that  $\chi(t)\in B_{\mathfrak q}$. Hence $t\in T(B_{\mathfrak q})$ and  $\bar t\in T(k_{\mathfrak q})$, the reduction of $t$ modulo ${\mathfrak q}$,  is well defined. The following is obvious from the construction. 

\begin{prop} \label{P:mini_torus} Assume ${\mathfrak q}\cap A\notin S$. 
Let $\Fr_{\mathfrak q}$ be the Frobenius generator of the decomposition subgroup $D_{\mathfrak q}\subseteq \Gal(E/K)$. Let $w=\varphi_C(\Fr_{\mathfrak q})$. 
 The element $\bar t$ is contained in the finite torus $T_w(k_{\mathfrak p}) \subset T(k_{\mathfrak q}) \subseteq G(k_{\mathfrak q})$. 
 \end{prop} 
 
   Starting with $t\in T(E)$ we have defined $\bar t\in G(k_{\mathfrak q})$ for all primes ${\mathfrak p}\notin S$, where ${\mathfrak q}$ is a prime ideal in $B$ such that  $\mathfrak q\cap A= {\mathfrak p}$. 
   Let $\bar C$ be the conjugacy class of $\bar t$ in $G(\bar k_{\mathfrak p})$. 
   By Proposition \ref{P:mini_torus} the action of $\Fr_{\mathfrak q}$  on $\bar t$ is  the action of $\varphi_C(\Fr_{\mathfrak q})\in W$ on $\bar t$. 
   Hence $\bar C$ is a $k_{\mathfrak p}$-rational conjugacy class. The characteristic polynomial $P_{\bar C}(x) \in k_{\mathfrak p}[x]$ is  the reduction of $P_C(x)$ modulo ${\mathfrak p}$. The class $\bar C$ does not depend on the choice of the ideal ${\mathfrak q}$. Indeed, if ${\mathfrak q}'$ is another prime ideal in $B$ such that  $\mathfrak q\cap A= {\mathfrak p}$
 then there exist $\sigma\in \Gal(E/K)$ such that $\sigma({\mathfrak q}')={\mathfrak q}$. 
  Replacing ${\mathfrak q}$ by ${\mathfrak q}'$ is equivalent to replacing $t$ by $t^{w}$ where $w=\varphi_C(\sigma)$.  
  Summarizing, the following definition is independent of the choice of $t$. 
 
 \begin{definition} For every ${\mathfrak p}\notin S$ 
 let $\bar C$ be the $G(\bar k_{\mathfrak p})$-conjugacy class of $\bar t$. The class $\bar C$ is called the reduction of $C$ modulo ${\mathfrak p}$.   
  \end{definition} 
 
 Of course, we know that the  class $\bar C$ is $k_{\mathfrak p}$-rational.
 
 \begin{prop} \label{torsion}
 Let $m$ be a positive integer.  Assume that the roots of 
$P_C(x)$ are not roots of 1.  Then there exists a finite set of primes $S_m \supseteq S$ 
such that for every ${\mathfrak p}\notin S_m$ the elements in $\bar C$, the reduction of $C$ modulo $\mathfrak p$,  do not have the order dividing $m$. 
\end{prop}  
\begin{proof}   
 By the assumption on $P_C(x)$, the polynomials $P_C(x)$ and $x^m-1$ are relatively prime. In particular, there exists polynomials $P(x)$ and $Q(x)\in K[x]$ such that 
\[ 
P(x) P_C(x)+ Q(x) (x^m-1)=1. 
\] 
Let  $S_m \supset S$ be a set of primes such that for every $\mathfrak p\notin S_m$ the coefficients of $P(x)$ and $Q(x)$ are in $A_{\mathfrak p}$. 
We can reduce the equation modulo every such prime. It follows that $\bar P_C(x)$ is relatively prime to $x^m-1$ in 
$k_{\mathfrak p}[x]$.  Hence the eigenvalues of $\bar t$ do not have the order dividing $m$. 
\end{proof}

  In view of Propositions \ref{P:mini_torus} and \ref{torsion} we have certain control of the order of elements in the class $\bar C$. 
 Let $K_{\mathfrak p}$ be the ${\mathfrak p}$-adic completion of $K$. Let $s\in C(K_{\mathfrak p})$. Let $\Lambda$ be a lattice in  
$V\otimes_K K_{\mathfrak p}$ such that $s\cdot \Lambda=\Lambda$. Let 
\[ 
\bar s:\Lambda /{\mathfrak p}\Lambda \rightarrow \Lambda/{\mathfrak p}\Lambda
\] 
 be the map induced by $s$. Let $|s|_{{\mathfrak p}'}$ be the prime to ${\mathfrak p}$-part of the order of $\bar s$. This is the order of the semi-simplification of $\bar s$. In particular, it does not depend on the choice of the lattice $\Lambda$. 
 For our applications we shall need the following:

\begin{prop} \label{P:order}  
Let ${\mathfrak p}\notin S$.  Let $s\in C(K_{\mathfrak p})$, where  $K_{\mathfrak p}$ be the ${\mathfrak p}$-adic completion of $K$.  Then $|s|_{{\mathfrak p}'}$ 
is equal to the order of  elements in $\bar C$, the reduction of $C$ modulo ${\mathfrak p}$. 
\end{prop} 
\begin{proof} It suffices to prove this statement after taking an unramified extension of $K_{\mathfrak p}$. In particular, we can extend by $E_{\mathfrak q}$,  the completion of $E$ at a prime $\mathfrak q\subset B$ such that $\mathfrak q\cap A=\mathfrak p$. Recall that  $C(K_s)$ contains $t \in T(E)\subset T(E_{\mathfrak q})$.  
Hence $s$ is conjugated to $t$ over a separable extension of $E_{\mathfrak q}$. However, since 
$T(E_{\mathfrak q})$ is a split torus, its Galois cohomology is trivial by the Hilbert Theorem 90. Thus $s$ is conjugated to $t$ over $E_{\mathfrak q}$. 
Now it suffices to prove the statement for $t$ and a lattice invariant under multiplication by $t$. Since the answer does not depend on the lattice 
we may as well choose the Chevalley lattice where $\bar t$ acts semi-simply.  Since $\bar t$ defines the class $\bar C$, the proposition follows. 
\end{proof}

\section{Galois groups of palindromic polynomials} 

Assume that $K$ is a field of characteristic 0. Let $P(x)\in K[x]$ be an irreducible palindromic polynomial of degree $2n$. Let  
 $x_{1}^{\pm 1}, \ldots , x_{n}^{\pm 1}$ be the roots of $P(x)$.
 Using the substitution
 $$
 y=x+\frac{1}{x}
 $$
 the polynomial can be reduced to a polynomial $Q(y)$ of degree $n$. If $y_{1}, \ldots ,
 y_{n}$ are its roots, then the roots of $P(x)$ are found by solving the equations
 $$
 x+ \frac{1}{x}=y_{i} \text{ for }i =1, \ldots ,n.
 $$
 Let $E$ and $F$ be the splitting fields of 
$P(x)$ and $Q(y)$. Let $\Gal(E/K)$ and $\Gal(F/K)$ be their Galois groups over $K$. 
Then $\Gal(F/K)\subseteq S_n$, where $S_n$ is the group of permutations of $y_1, \ldots , y_n$ and  $\Gal(E/F)\subseteq C_2^n$, the elementary $2$-group of order 
$2^n$ which acts by permuting $x_i$ and $1/x_i$. 
 The group $\Gal(E/K)$ is contained in a semi-direct product of $C_2^n$ and $S_n$, the Weyl group of type $BC_n$.  Let 
$\Delta$ be the discriminant of $Q(y)$. Let $\epsilon$ be the sign character $S_n$. By the 
restriction and inflation we view $\epsilon$ as the character of $\Gal(E/K)$. The group 
$\Gal(E/K)$  has another quadratic character $\epsilon'$ defined as follows. Let 
\[ 
\Delta'=\prod_{i=1}^n (x_i-{x}^{-1}_i)^2. 
\] 
Then we define $\sigma(\sqrt{\Delta'})=\epsilon(\sigma)\sqrt{\Delta'}$ for every $\sigma$ in 
$\Gal(E/K)$. Note that the restriction of $\epsilon'$ to $\Gal(E/F)\subseteq C_2^n$ is given by 
$\epsilon'(a_1, \ldots, a_n)=\prod_{i=1}^n a_i$. Usefulness of $\epsilon'$ lies in the fact 
that $\Delta'$ can be easily computed. Indeed, 
\[ 
\Delta'=\prod_{i=1}^n (x_i-{x}^{-1}_i)^2=\prod_{i=1}^n (y_i^2-4)^2=Q(2)\cdot Q(-2) . 
\]

 \begin{prop} \label{P:separable} Let $P(x)\in K[x]$ be a palindromic polynomial of degree $2n$. Let $Q(y)$ be the polynomial of degree $n$ obtained by the palindromic reduction from $P(x)$. Let $\Delta$ be the discriminant of $Q(x)$.
The polynomial $P(x)$ is separable if and only if $\Delta \neq 0$, and $\Delta'\neq 0$. 
\end{prop} 
\begin{proof} The roots of  of $P(x)$ come in pairs $(x_i, 1/x_i)$. If $P(x)$ is not separable then $x_i=1/x_i$ for some $i$ or 
$x_i= x_j$ (or $=1/x_j$) for $i\neq j$. If the first case $x_i=\pm 1$ and $2$ or $-2$ is a root of $Q(y)$. In the second case $y_i=y_j$, hence 
$\Delta =0$. 
\end{proof} 

Assume now that $K$ is a number field and $A$ its ring of integers. If ${\mathfrak p}$ is a maximal ideal, let $A_{\mathfrak p}$ be the localization of $A$ by $A\setminus {\mathfrak p}$. 

\begin{cor} \label{C:ramification} Assume that $K$ is a number field. Let ${\mathfrak p}$ be a prime such that  $P(x)\in A_{\mathfrak p}[x]$. If $\Delta$ and $\Delta'$ are in  $A_{\mathfrak p}^{\times}$  then the splitting field of $P(x)$ is unramified at ${\mathfrak p}$.  
\end{cor}

\vskip 5pt 

Assume now that $K$ is a totally real number field and  that all roots of $Q(y)$ are in the interval $(-2,2)$. Then, and only then, 
the roots of $P(x)$ are pairs of complex conjugates on the unit circle. Note that the complex conjugation is given by 
$c=(-1,\ldots , -1)\in C_2^n$, and it is a central element in $\Gal(E/K)$.

\begin{prop} \label{palin} Assume that $K$ is a totally real field. 
Assume that the Galois group of $Q(y)$ is $S_n$, and the roots  are in the interval $(-2,2)$. Assume that $\Delta'$ is not a square in $K$ 
and $K(\sqrt{\Delta}) \neq K(\sqrt{\Delta'})$. Then the Galois group of $P(x)$ is 
isomorphic to the Weyl group of type $BC_n$ or, if $n$ is odd, to the 
direct product $\langle c\rangle\times S_n$ where $c$ is the complex conjugation.  
\end{prop}
\begin{proof} If the restriction of $\epsilon'$ to $\Gal(E/F)$ is trivial, then it induces a non-trivial character of $S_n$. But $\epsilon$ is the unique character of $S_n$, 
thus $\epsilon'=\epsilon$, a contradiction to the assumption on the quadratic fields. 
Now we can finish easily. Note that $C_2^n$ is $S_n$-generated by any element outside the 
kernel of $\epsilon'$ except when $n$ is odd and the element is $c$. Thus either $\Gal(E/F)=C_2^n$ and $\Gal(E/K)$ is the Weyl group of type $BC_n$ or 
$\Gal(E/K)$ is the extension of $S_n$ by $\langle c\rangle$. But this extension must split, as given by $\epsilon'$. 
\end{proof}

Let $P_i(x)\in K[x]$, $i=1,2$ be two palindromic polynomials satisfying the conditions of 
 Proposition \ref{palin}, and $\Delta_i$, $\Delta'_i$, the two discriminants. 

\begin{cor} \label{Cor1}
Let $E_1$ and $E_2$,  be the splitting fields of $P_1[x]$ 
and  $P_2[x]$, respectively. Then $E_1$ and $E_2$ 
are algebraically independent over $K$ if and only if the bi-quadratic fields 
$K(\sqrt{\Delta_1}, \sqrt{\Delta_1'})$ and $K(\sqrt{\Delta_2}, \sqrt{\Delta_2'})$ are. 
\end{cor} 
\begin{proof} Since $E_1$ and $E_2$ are Galois, it suffices to show that $E_1\cap E_2=K$. 
Note that $E_1\cap E_2$ is Galois in both $E_1$ and $E_2$. 
One easily sees that any non-trivial normal subgroup 
of the semi-direct product of $C_2^n$ and $S_n$ (and of $\langle c\rangle\times S_n$) 
is contained in a kernel of one of the three 
characters: $\epsilon,\epsilon'$ and $\epsilon\epsilon'$. Thus $E_1\cap E_2$, if strictly bigger than 
$K$, must contain a quadratic field common to $K(\sqrt{\Delta_1}, \sqrt{\Delta_1'})$ and $K(\sqrt{\Delta_2}, \sqrt{\Delta_2'})$. 
\end{proof}

For applications to $G_2$ we note the following corollary. 
 
\begin{cor} \label{Cor2}
Assume that $n=3$, $P(x)$ satisfies the conditions of Proposition \ref{palin}, 
 and that the roots $x_1^{\pm 1},x_2^{\pm 1},x_3^{\pm 1}$ of $P(x)$ satisfy 
$x_1 x_2 x_3=1$. Then the Galois group of $P(x)$ is isomorphic 
to $\langle c\rangle\times S_3\cong D_6$, the dihedral group of order $12$. 
\end{cor} 
\begin{proof} Indeed, the relation between the roots implies that the degree of 
$E$ cannot be $48$. Thus the Galois group must be $\langle c\rangle\times S_3$.
\end{proof}

\section{Generating $G_2$}

We realize the Chevalley group of type $G_2$ on its 7-dimensional representation $V_7$.  In particular, a conjugacy class 
gives rise the a characteristic polynomial $P_C(x)$ of degree 6. It is a palindromic polynomial  whose roots satisfy the properties as in 
Corollary \ref{Cor2}. 

\vskip 5pt 

Maximal subgroups of $G_2(p)$ have been classified by Aschbacher.  We extract the information from Corollary 11 on page 199 in \cite{As}.

\begin{theorem} Assume $p>3$. 
The maximal subgroups of $G_2(p)$ are as follows.
\begin{enumerate}
\item maximal parabolic subgroups.
\item $SL_3(p).2$ and $SU_3(p).2$.
\item $SO_{4}^{+}(p)$.
\item $PGL_2(p)$ if $p>5$, acting on $V_{7}$ like on homogeneous polynomials in two 
variables of degree $6$.
\item $2^3.L_3(2)$, the stabilizer of an orthonormal basis of $V_{7}$; the order is $2^{6}\cdot 3\cdot 7$. 
\item $L_2(13)$ if $\bbF_p$ is a splitting field for $T^2 -13$; the order is $2^{2}\cdot3\cdot 7\cdot 13$.
\item $G_2(2)$; the order is $2^{6}\cdot 3^{3}\cdot 7$. 
\item $L_2(8)$ if  $\bbF_p$ is a splitting field for $T^2 -3T+1$; the order is $2^{3}\cdot 3^{2}\cdot 7$. 
\item $J_1$ if $p =11$; the order is $2^{3}\cdot 3\cdot 5\cdot 7\cdot 11\cdot 19$. 
\end{enumerate}
\end{theorem}

\smallskip

\begin{cor} \label{C:generators} Assume that $p>3$. 
Let $u$ and $t$ be two elements in $G_2(p)$ of orders  $>3$. If 
 the order of $u$ divides $p^2+ p + 1$ and the order of $t$ divides $p^2-p+1$ 
then  $u$ and $t$ are not contained in any maximal subgroup except, perhaps, the five
groups  of bounded order labeled (5) - (9). 
\end{cor}
\begin{proof} The proof is based on the theorem of Lagrange. Let $\Phi_n(x)$ be the $n$-th cyclotomic polynomial. In particular, 
\[ 
\Phi_3(p)=p^2+p+1\text{ and } \Phi_6(p)=p^2-p+1. 
\] 
 For any finite group $G$, let 
$|G|_{p'}$ be the prime to $p$ part of the order of $G$. 
Relevant to us are the following orders: 
\[ 
\begin{array}{|c||c|c|c|c|}
\hline  G  & SL_3(p)  & SU_3(p)  & SO_4^+(p) & GL_2(p) \\
\hline  |G|_{p'} & \Phi_1(p)^2 \Phi_3(p)  &  \Phi_1(p)\Phi_2(p) \Phi_6(p) & \Phi_1(p)^2\Phi_2(p)^2 &  \Phi_1(p)^2\Phi_2(p) \\
\hline
\end{array} 
\]  
Let $n=p^2+ap+1$ and $m=p^2+ap+1$ where $a,b$ are integers. Then $n-m=(a-b)p$. Hence the gcd of $n$ and $m$ is a divisor of $a-b$. 
This observation implies that $\Phi_1(p)^2$, $\Phi_2(p)^2$, $\Phi_3(p)$ and $\Phi_6(p)$ are essentialy pairwise relatively prime. Note that 
$\Phi_3(p)$ and $\Phi_6(p)$ are also odd.  Hence, by the theorem of Lagrange,  $t$ is contained in $SU_3(p). 2$ and in no other maximal 
subgroups labeled (1)-(4). However,  $u$ cannot be contained in $SU_3(p).2$. This proves the corollary. 
\end{proof} 

\begin{theorem} \label{main2}
Let $A$ and $B$ be two semi-simple, regular, rational conjugacy classes of $G_2(\mathbb Q)$. Assume that 
 the splitting fields of the characteristic polynomials $P_A(x)$ and $P_B(x)$  are algebraically independent and the Galois groups are isomorphic to $D_6$. 
Assume that for almost all primes we are given $a_p\in A(\mathbb Q_p)$,  $b_p\in B(\mathbb Q_p)$ and a maximal compact  
subgroup $U_p$ in $G_2(\mathbb Q_p)$ containing $a_p$ and $b_p$. Then, for a set of primes of density at least 1/18, the 
group $U_p$ is hyperspecial and the projections  of $a_p$ and $b_p$  generate the reductive quotient  of $U_p$ isomorphic to $G_2(p)$. 
\end{theorem} 
\begin{proof} Let $\Lambda$ be a lattice, in the 7-dimensional representation of $G_2$, preserved by $U_p$. (See \cite{GY} for a beautiful description of maximal compact subgroups in $G_2(\mathbb Q_p)$ as stabilizers of orders in the octonion algebra.) Then $\Lambda/p\Lambda$ is a $U_p$-module. The group 
$U_p$ acts on a semi simplification of $\Lambda/p\Lambda$ through its reductive quotient: $G_2(p)$, $SL_3(p)$ or $SO_4(p)^+$. 
 Let $\bar a_p$  be the projection of $a_p$ to the reductive quotient of $U_p$. 
  Let $\bar A$  be the reduction modulo $p$ of the rational conjugacy class $A$.  By Proposition \ref{P:order} the prime to $p$-part of the order of 
$\bar a_p$ is the same as the order of elements in $\bar A$. 
 By Proposition \ref{P:mini_torus}, the order of elements in 
$\bar A$ divides the order of the finite torus $T_w$ where $w=\varphi_A(\Fr_p)$.  
If $\varphi_A(\Fr_p)=6a$,   the Coxeter  class, then the order elements in $\bar A$ 
divides $1-p+p^2$.  In addition, the order of  elements in $\bar A$ can be arranged to be non-trivial, in fact  as large as needed, by Proposition \ref{torsion}. 
This forces the reductive quotient of $U_p$ to be isomorphic to $G_2(p)$, i.e. $U_p$ is hyperspecial. Similarly if, in addition,  
 $\varphi_B(\Fr_p)=3a$ then 
$\bar a_p$ and $\bar b_p$ generate the quotient $G_2(p)$ by Corollary \ref{C:generators}. 
 By \v Cebotarev, $\varphi_A(\Fr_p)=6a$ for a set of primes of density 1/6 
and, independently,  $\varphi_B(\Fr_p)=3a$ for a set of primes of density 1/6.  Since the roles of $A$ and $B$ can be reversed, we get $G_2(p)$ for a set of primes of density 
\[
\frac{1}{6}\cdot \frac{1}{6} + \frac{1}{6}\cdot \frac{1}{6} = \frac{1}{18}. 
\] 
\end{proof}

\section{Representations attached to automorphic forms} 

Let $\bbA$ be the ring of adel\'es of $\bbQ$ and $\Pi$ a  cuspidal automorphic representation of $GL_{m}(\bbA)$ where $m=2n+1$. 
 Fix a prime $q$ and assume that $\Pi$ is unramified for all primes $\ell\neq q$.
 Let $R_{\ell}(x)$ denote the characteristic polynomial of the Satake parameter of 
 the local component $\Pi_{\ell}$.  Assume that $\Pi$ satisfies the following properties:
 \begin{enumerate}
 \item The infinitesimal character of $\Pi_{\infty}$ is the infinitesimal character of 
 the trivial representation of $GL_{m}(\bbR)$. 
 \item  $\Pi_{q}$ is the Steinberg representation. 
 \item $R_{\ell}(x)=P_{\ell}(x)\cdot (x-1)$, and $P_{\ell}(x)$ is a palindromic 
 polynomial with coefficients in $\bbZ[\frac{1}{\ell}]$. 
 \end{enumerate}
 These conditions mean that the local components are lifts from $Sp_{2n}$. In particular, $\Pi$ is self dual. Thus, by a result of Harris and Taylor \cite{HT}, for every prime 
 $p$ there exists a continuous representation 
  \[ 
 \rho_{p}: \Gal(\bar\bbQ/\bbQ) \rightarrow GL_{m}(\bar{\bbQ}_p)
 \]
 unramified at all primes $\ell\neq p,q$ such that $R_{\ell}(x)$ is the characteristic  polynomial of $\rho_{p}(\Fr_{\ell})$, where $\Fr_{\ell}$ is the Frobenius
 at $\ell$. Moreover, if $p\neq q$ then the image of the inertia subgroup $I_{q}$
 contains the regular unipotent class and  $\rho_p$ is irreducible by a result of Taylor and Yoshida \cite{TY}. 
 
 \begin{prop} If $p\neq q$ then  $\rho_{p}$ is defined over $\bbQ_{p}$, that is, we have 
   \[ 
 \rho_{p}: \Gal(\bar\bbQ/\bbQ) \rightarrow GL_{m}({\bbQ}_p). 
 \]
 \end{prop} 
 \begin{proof} Let $\Gamma$ be the image of $\rho_p$. Consider the group algebra $A=\mathbb Q_p[\Gamma]\subseteq M_{m}(\bar{\bbQ}_p)$.
 Since $\rho_{p}$ is irreducible, the algebra $A$ is simple. Hence $A$ is isomorphic to $M_r(D)$ where $D$ is a division algebra. The center of $D$ is equal to the field of reduced traces of $A$.  The reduced trace is simply the restriction to $A$ of the usual trace on $M_{m}(\bar{\bbQ}_p)$. Hence, 
 by the assumption (3) on $\Pi$, the field of reduced traces of $A$ is ${\bbQ}_p$. This implies that $D$ is a central simple algebra over ${\bbQ}_p$. The algebra $M_r(D)$ acts on $D^r$ from the left. This action commutes with the action of $D$ on $D^r$ from the right. Let $\sigma\in M_r(D)$ be an element of order 2. We can decompose $D^r$ as a sum of the two eigenspaces for $\sigma$ with eigenvalues $\pm 1$. Each  of these two eigenspaces is a $D$-module for the right  action of $D$. Hence the (reduced) trace of $\sigma$ is a multiple of the degree of $D$. However, in \cite{Ta} Taylor has shown that $Tr(\rho_p(c))=\pm 1$, where $c$ is the complex conjugation. Hence 
 $D=\bbQ_p$ and $m=r$, that is, $\Gamma \subseteq GL_m(\bbQ_p)$. 
 \end{proof}

 \begin{prop} If $p\neq q$ then the image of $\rho_{p}$ is contained in a 
 split orthogonal group $SO_{m}(\bbQ_{p})$. 
 \end{prop}
 \begin{proof} 
 Since $\rho_{p}$ is irreducible and self-dual it preserves a non-degenerate 
 bilinear form,  unique up to a non-zero scalar. Since $m$ is odd, the form has to be orthogonal. Moreover, 
 since the determinant of $\rho_{p}(\Fr_{\ell})$ is one for all $\ell\neq p,q$, and these 
 elements are dense, the image of the group must be contained in $SO_{m}(\bbQ_{p})$. 
 There are two isomorphism classes of odd orthogonal groups over $\bbQ_{p}$, but only 
 the split isomorphism class contains the regular unipotent conjugacy class. 
 \end{proof}
 
 Assume now that $n=3$. For every $\ell\neq q$ let
 let $Q_{\ell}(y)=y^{3}-a_{\ell}y^{2}+b_{\ell}y-c_{\ell}$
 be the polynomial obtained by the palindromic reduction of $P_{\ell}(x)$. Let 
 $x_{1}^{\pm 1},x_{2}^{\pm 1}, x_{3}^{\pm 1}$ be the roots of $P_{\ell}$. 
 If $\Pi_{\ell}$ is a local lift from $G_{2}(\bbQ_{\ell})$ then $x_{1}x_{2}x_{3}=1$. 
 This imposes a condition on the coefficients of $P_{\ell}(x)$ which translates to 
 \[
 a_{\ell}^{2}=c_{\ell}+2b_{\ell}+4. 
 \]
 If this holds for every $\ell\neq q$, we shall say that $\Pi$ is locally a lift from $G_{2}$. 
 
 \begin{cor} \label{C:G2} 
 With $n=3$. If $\Pi$ is locally a lift from $G_{2}$ then the image of 
 $\rho_{p}$ is contained in $G_{2}(\bbQ_{p})$ for all $p\neq q$. 
 \end{cor}
 \begin{proof}
 Let $G$ be the Zariski closure of the image of $\rho_{p}$ in 
 $SO_{7}(\bbQ_{p})$. Since $G$ acts irreducibly on the 7-dimensional representation, 
 $G$ is a reductive group. 
  Since $a^{2}=c^{2}+2b+4$ is an algebraic condition, and 
 $\rho_{p}(\Fr_{\ell})$ are dense in the image, this condition holds for all elements 
 in $G$. Thus the rank of $G$ is at most 2.  Recall that the image of the inertia $I_{q}$ 
 contains the regular unipotent element. Since $1$ is contained in the closure of 
 any unipotent class, it follows that the connected component $G^{\circ}$ of 1 
 contains the regular unipotent class. Thus $G^{\circ}$ is either $G_{2}(\bbQ_{p})$ or 
 its principal $PGL_{2}(\bbQ_{p})$. Since these two groups are self-normalizing in 
 $SO_{2n+1}(\bbQ_{p})$, it follows that $G\cong G_{2}(\bbQ_{p})$ or 
 $PGL_{2}(\bbQ_{p})\subseteq G_{2}(\bbQ_{p})$. 
 \end{proof}

\section{An Example}

Assume that $G$ is the unique form, over $\bbQ$, of the exceptional Lie group of type $G_{2}$
such that $G(\bbR)$ is compact and $G(\bbQ_{p})$ is split for all primes $p$. 
In \cite{GS} it is shown that there exists an automorphic representation 
$\pi$ such that $\pi_{\infty} \cong \bbC$, $\pi_5$ is the Steinberg representation, and 
$\pi_{\ell}$ is unramified for all other primes $\ell$. Moreover, the characteristic polynomial  $R_{\ell}(x)$ 
 of the Satake parameter $s_{\ell}\in G_2(\mathbb C)$ of $\pi_{\ell}$, 
acting on the 7-dimensional representation has coefficients in $\bbZ[\frac{1}{\ell}]$ (see also \cite {Gr}). 
In \cite{LP} Lansky and Pollack have calculated the polynomial $R_{\ell}(x)$ for $\ell=2$ and $3$: 
\[
R_{2}(x)=x^{7}+\frac{1}{4}x^{6}-x^{5} -\frac{13}{16}x^{4}
+\frac{13}{16}x^{3} +x^{2} -\frac{1}{4}x-1
\]
\[
R_{3}(x)=x^{7}-\frac{29}{3^{3}}x^{6}+\frac{175}{3^{5}}x^{5} -\frac{1099}{3^{6}}x^{4}
+\frac{1099}{3^{6}}x^{3} -\frac{175}{3^{5}}x^{2} +\frac{29}{3^{3}}x-1.
\]
After factoring $R_{\ell}(x)=P_{\ell}(x)\cdot(x-1)$, the two palindromic polynomials
$P_{\ell}(x)$ are reduced to 
\[
Q_{2}(y)= y^{3}+\frac{5}{4}y^{2}-\frac{11}{4}y-\frac{49}{16}
\]
\[
Q_{3}(y)= y^{3}-\frac{2}{3^{3}}y^{2}-\frac{572}{3^{5}}y-\frac{520}{3^{6}}.
\]
Let $\Delta_{\ell}$ be the discriminant of $Q_{\ell}$. We have the following 
numerical values:
\smallskip 
\[
\begin{array}{|c|c|c|c|}
\hline 
\ell & Q_{\ell}(2) & Q_{\ell}(-2) & \Delta_{\ell} \\
 \hline 
2 &\frac{71}{16} & -\frac{9}{16} & \frac{71\cdot 199}{2^{8}}\\
\hline 
3 & \frac{2^{7}\cdot 13}{3^{6}} & -\frac{2^{6}\cdot 7^{2}}{3^{6}} & \frac{2^{14}\cdot 13\cdot 7321}{3^{16}}\\
\hline 
\end{array}
\]

\begin{prop} \label{P:tempered}  The local components $\pi_2$ and $\pi_3$ are tempered. The splitting fields of $P_2(x)$ and $P_3(x)$ have the Galois group 
isomorphic to $D_6$ and are algebraically independent.
\end{prop} 
\begin{proof} Since $\Delta_{\ell}>0$, the polynomials $Q_{\ell}(y)$ have 3 real roots, each. 
Since $Q_{\ell}'( -2)>0$,  $Q_{\ell}'(0)<0$ and $Q_{\ell}'( 2)>0$ the inflection points are in the segment $(-2,2)$. 
Since $Q_{\ell}(-2)<0$ and $Q_{\ell}(2)>0$ the roots are in  the segment $(-2,2)$. This shows that the roots of 
$P_{2}(x)$ and $P_{3}(x)$ lie on the unit circle. 

Since $\Delta_2$ and $\Delta_3$ are not rational squares, it follows that the Galois group of, both, 
$Q_2(y)$ and $Q_3(y)$ is $S_3$. By Corollary \ref{Cor2}, the Galois group of, both, $P_2(x)$ and $P_3(x)$ 
is isomorphic to $D_6$. Moreover, the two splitting fields are algebraically independent by Corollary \ref{Cor1}. 

\end{proof} 

In \cite{GS} it is also shown that $\pi$ lifts to a cuspidal  automorphic representation $\sigma$ on $Sp_6$ such that 
 \begin{enumerate}
 \item $\sigma_{\infty}$ is a holomorphic discrete series representation. 
 \item  $\sigma_{5}$ is the Steinberg representation. 
 \item $\sigma_{\ell}$ is an unramified representation, a lift from $G_2(\mathbb Q_{\ell})$, for $\ell\neq 5$. 
 \item $\sigma_{2}$ and $\sigma_3$ are tempered with Satake parameters given by $R_2(x)$ and $R_3(x)$. 
 \end{enumerate}
\vskip 5pt 

We shall now use the results of Arthur \cite{Ar} to lift $\sigma$ to a cuspidal form on $GL_7$. We need the following. 

\begin{prop} Let $\sigma$ be a cuspidal automorphic representation on $Sp_{2n}$ such that $\sigma_q$ is the Steinberg representation for a prime $q$. 
Let $\Pi$ be the automorphic representation of $GL_{2n+1}$, the lift of $\sigma$ as in Theorem 1.5.2 in \cite{Ar}. Then $\Pi_q$ is the Steinberg representation and  $\Pi$ is cuspidal.  
\end{prop} 
\begin{proof} The representation $\Pi$ belongs to the automorphic $L^2$-spectrum.  We shall now argue that $\Pi$ is a cuspidal automorphic representation if 
a local component is the Steinberg representation. Recall that the Steinberg representation is Whittaker-generic. Thus,  it cannot be a local component of an automorphic representation in the residual spectrum.  Next, the Steinberg representation is not induced from a proper parabolic subgroup. Thus, it cannot be a local component of an automorphic representation in the continuous spectrum. Hence the Steinberg representation can only be a local component of a cuspidal automorphic representation. In order to prove the proposition it remains to show that $\Pi_q$ is the Steinberg representation. 

Write $G=Sp(2n)$. Let F be a $p$-adic local field and $W_F$ the Weil group.  Let $\Psi(G)$ be the set of Arthur parameters i.e. maps
\[ 
\psi : W_F \times  SU(2) \times  SU(2) \rightarrow SO(2n+1,\bbC) 
\] 
with bounded image. In Theorem 1.5.1 in \cite{Ar}, to every such $\psi$, Arthur attaches a set (a packet) of unitary representations of $G$ over $F$, 
in fact, local constituents of automorphic representations.  By the strong approximation
theorem, the trivial representation appears only as a component of the 
global trivial representation. Hence the trivial representation appears only in the packet of  $\psi$ trivial on $W_F$, trivial  on the first $SU(2)$, and such that 
the image of the second $SU(2)$ is the principal $SU(2)$ in $SO(2n+1,\bbC)$. 
 By Lemma 7.1.1 in \cite{Ar}, the involution on the set of  parameters  $\psi$ given by the switching of the two $SU(2)$ corresponds to the Aubert involution on
representations.  Since the Aubert involution of the trivial representation is the Steinberg representation, 
 the Steinberg representation is contained in its $L$-packet and in no other packets. 

The local components of $\sigma$ sit in the packets determined by the local components of $\Pi$. 
 However, since we do not know the generalized Ramanujan conjecture, the local components of $\sigma$ a priori lie in 
packets parameterized by a larger set, denoted by $\Psi^+_{\unit}(G)$ by Arthur. According to the displayed formula (1.5.1) in \cite{Ar}  the
packets in $\Psi^+_{\unit}(G)$, but not in $\Psi(G)$,  consist of representations
fully induced from a proper parabolic subgroup. In particular, the Steinberg representation cannot be in any of these packets.
 Thus, if $\sigma_q$ is the Steinberg representation, then $\Pi_q$  is  the Steinberg representation of $GL_{2n+1}$ by
the above discussion. Hence  $\Pi$ is cuspidal.

\end{proof} 

Since the lift $\Pi$ is cuspidal, it is a functorial lift of $\sigma$. Summarizing, there exits a cusp form on $GL_7$ to which we can apply Corollary \ref{C:G2}. Thus, the image of the associated $p$-adic representation is contained in $G_2(\mathbb Q_p)$ for all $p\neq 5$. 
Applying Theorem   \ref{main2} and Proposition \ref{P:tempered} yields the following result:  

\begin{theorem} There exists an extension of $\bbQ$ with $G_2(p)$ as the Galois group, unramified at $5$ and $p$ only, for a set of primes $p$ of density at least $1/18$. 
\end{theorem}


\begin{thebibliography}{999}

\bibitem[Ar]{Ar} J. Arthur, {\em The endoscopic classification of representations: Orthogonal and Symplectic Groups.} Colloquium Publications, {\bf 61}, 2013, AMS. 


\bibitem[As]{As} M. Aschbacher, 
{\em Chevalley groups of type $G_2$ as the group of a trilinear form.} 
J. of Algebra  {\bf 109} (1987), no 1, 193--259. 

\bibitem[Ca]{Ca} R. W. Carter, {\em Finite Groups of Lie Type.} Wiley \& Sons, 1993. 




\bibitem[Du]{Du} L. V. Dieulefait, {\em On the images of the Galois representations attached to genus 2 Siegel module forms.} 
J. Reine Angew. Math. {\bf 553} (2002), 183-2000. 

\bibitem[GY]{GY} W. T. Gan and J. K. Yu,
{\em Schemas en groupes et immeubles des groupes exceptionels sur un corp locale. 
Premiere parte: le groupe $G_{2}$.} Bull. Math. Soc. France {\bf 131} (2003), 307--358. 

\bibitem[Gr]{Gr} B. Gross, 
{\em Algebraic modular forms.} 
Isr. J. Math. J. {\bf 113} (1999), 61--93.

\bibitem[GS]{GS} B. Gross and G. Savin, 
 {\em Motives with Galois group of type $G_2$.}
Compositio Math.  {\bf 114} (1998), 153--217.

\bibitem[HT]{HT} M. Harris and R. Taylor, 
{\em The geometry and cohomology of some simple Shimura varieties.} Annals of Mathematics Studies, vol {\bf 151} Princeton University Press, Princeton, NJ, 2001. 

\bibitem[LP]{LP} D. Pollack and J. Lansky, 
 {\em Hecke algebras and automorphic forms.}
 Compositio Math. {\bf 130} (2002), 21--48. 



\bibitem[Ta]{Ta} R. Taylor, 
{\em The image of complex conjugation in $\ell$-adic representations associated to automorphic forms.} 
Algebra and Number Theory {\bf 6} (2012),  405--435. 

\bibitem[TY]{TY} R. Taylor and T. Yoshida,
{\em Compatibility of local and global Langlands correspondences.} 
J. Amer. Math. Soc. {\bf 20} (2007), no 2, 467--493.  


\end{thebibliography}
\end{document}